\documentclass[11pt,a4]{article}

\usepackage[utf8]{inputenc}
\usepackage[english,russian]{babel}

\usepackage{graphicx}
\usepackage{epstopdf}
\usepackage{amsmath}
\usepackage{amsfonts, amsthm}
\usepackage{amssymb}
\usepackage{url}

\usepackage[framemethod=1]{mdframed}
\usepackage{mathtools}
\usepackage{makecell}
\usepackage{tikz}
\usepackage{dsfont}

\newcommand{\norm}[1]{\left\lVert#1\right\rVert}
\newtheorem{theorem}{Теорема}
\newtheorem{lemma}{Лемма}
\newtheorem{corollary}{Следствие}
\newtheorem{definition}{Определение}
\newtheorem{remark}{Замечание}
\newtheorem{statement}{Утверждение}
\newcommand\myeq{\mathrel{\stackrel{\makebox[0pt]{\mbox{\normalfont\tiny def}}}{=}}}
\newcommand{\leqarg}[1]{\ensuremath{\stackrel{\text{#1}}{\leq}}}
\newcommand{\eqarg}[1]{\ensuremath{\stackrel{\text{#1}}{=}}}
\newcommand\abs[1]{\left|#1\right|}

\DeclareMathOperator*{\argmin}{arg\,min}

\newcommand*\circled[1]{\tikz[baseline=(char.base)]{
		\node[shape=circle,draw,inner sep=2pt] (char) {#1};}}
\allowdisplaybreaks

\begin{document}

\begin{center}
\bf
Адаптивный быстрый градиентный метод в задачах стохастической оптимизации

Adaptive fast gradient method in stochastic optimization tasks
\end{center}

\begin{center}
\textit{
Александр Игоревич Тюрин
\footnote{atyurin@hse.ru, Национальный исследовательский университет «Высшая школа экономики», 101000, Россия, г. Москва, ул. Мясницкая, д. 20 \\
Alexander Tyurin, National Research University Higher School of Economics, Myasnitskaya str. 20, Moscow, Russia 101000
}}
\end{center}

\begin{abstract}
В данной работе приводится стохастический адаптивный ускоренный градиентный метод на основе зеркального варианта метода подобных треугольников. На сколько мы знаем, это первая попытка добавления адаптивности в стохастический метод. Ко всему прочему, главный результат приводится в виде оценки вероятностей больших уклонений.

In this paper, we describe a stochastic adaptive fast gradient descent method based on the mirror variant of similar triangles method. To our knowledge, this is the first attempt to use adaptivity in stochastic method. Additionally, a main result was proved in terms of probabilities of large deviations.
\end{abstract}


\textbf{Ключевые слова} Быстрый Градиентный Метод, Композитная оптимизация, Стохастическая Оптимизация, Адаптивная Оптимизация


\textbf{УДК} 519.85

\section{Введение}
В данной работе представлен алгоритм стохастической оптимизации, который является обобщением быстрого градиентного метода \cite{nesterov2013introductory}. Методы стохастической оптимизации в последнее время являются популярными по той причине, что они позволяют уменьшать сложность подсчета градиента, что является очень важным, так как существуют примеры функций, в которых невозможно за разумное время подсчитать градиент оптимизируемой функции хотя бы в одной точке. Помимо этого, в задачах стохастической оптимизации \cite{gasnikov2016netrivialnosti} в силу формулировки самой задачи единственный разумный способ получить направление спуска -- это сэмлирование векторов, которые имеют несмещенные по отношению к истинному градиенту направления. Данный подход также популярен в глубинном обучении \cite{krizhevsky2012imagenet}, в задачах, которые имеют вид суммы функций \cite{gasnikov2016netrivialnosti}. Стоит отметить, что сложность подсчета стохастического градиента, как правило, можно регулировать. Начиная с того, что можно считать честно полностью градиент, и заканчивая тем, что можно брать случайную грубую оценку настоящего градиента, но при этом стоит понимать, что дисперсия данной оценки может довольно сильно влиять на скорость сходимости \cite{lan2012optimal,devolder2013exactness}. В данной работе предлагается использовать технику mini-batch, которая позволяет агрегировать случайные несмещенные оценки в достаточном количестве, при этом данная техника приводит к оптимальным оценкам. Ко всему прочему, данный метод является адаптивным \cite{gasnikov2018universal} и поддерживает наличие композита \cite{lan2012optimal}.

\section{Адаптивный зеркальный вариант метода подобных треугольников с неточным выборочным стохастическим $(\delta, L)$-оракулом} \label{sec:mmtDLST}

Опишем сначала общую постановку задачи выпуклой оптимизации \cite{nesterov2010introductory}. 
Пусть определены функции $f(x): Q \longrightarrow \mathds{R}$ и $h(x): Q \longrightarrow \mathds{R}$ и дана произвольная норма $\norm{\cdot}$ в $\mathds{R}^n$. Обозначим $F(x)$, как $F(x) \myeq f(x) + h(x)$. Сопряженная норма определяется следующим образом:
\begin{gather}\norm{\lambda}_* \myeq \max\limits_{\norm{\nu} \leq 1;\nu \in \mathds{R}^n}\langle \lambda,\nu\rangle,\,\,\,\forall \lambda \in \mathds{R}^n.\end{gather}
Будем полагать, что
\begin{enumerate}
	\item $Q \subseteq \mathds{R}^n$, выпуклое, замкнутое.
	\item $f(x)$ и $h(x)$ -- непрерывные и выпуклые функции на $Q$.
	\item $f(x)$ -- гладкая функция.
	\item $F(x)$ ограничена снизу на $Q$ и достигает своего минимума в некоторой точке (необязательно единственной) $x_* \in Q$.
\end{enumerate}

Рассмотрим следующую задачу оптимизации:
\begin{align}
\label{mainTask3}
F(x) \rightarrow \min_{x \in Q}.
\end{align}

Введем два понятия: прокс-функция и дивергенция Брэгмана \cite{gupta2008bregman}.
\begin{definition}
$d(x):Q \rightarrow \mathds{R}$ называется прокс-функцией, если $d(x)$ непрерывно дифференцируемая на $\textnormal{int }Q$ и $d(x)$ является 1-сильно выпуклой относительно нормы $\norm{\cdot}$ на $\textnormal{int }Q$.
\end{definition}
\begin{definition}
Дивергенцией Брэгмана называется 
\begin{align}
V(x,y) \myeq d(x) - d(y) - \langle\nabla d(y), x - y\rangle,
\end{align}
где $d(x)$ -- произвольная прокс-функция.
\end{definition}
Легко показать \cite{ben2001lectures}, что \begin{gather}V(x,y) \geq \frac{1}{2}\norm{x - y}^2.\label{bregmanIneq}\end{gather}

\begin{definition}
\leavevmode
\label{defdeltaLstoch}
Стохастическим $(\delta, L)$-оракулом будем называть оракул, который на запрашиваемую точку $y \in Q$ дает пару $\left(f_\delta(y), \nabla f_\delta(y;\xi)\right)$ такую, что
\begin{gather}
\label{exitLDLSTOrig}
0 \leq f(x) - f_\delta(y) - \langle\nabla f_\delta(y), x - y\rangle \leq \frac{L}{2}\norm{x - y}^2 + \delta ,\,\,\, \forall x \in Q,
\end{gather}
\begin{gather}
\label{ST1}
\mathbb{E}\left[\nabla f_\delta(y;\xi)\right] = \nabla f_\delta(y),\,\,\, \forall y \in Q,
\end{gather}
\begin{gather}
\label{ST2}
\mathbb{E}\left[\exp\Bigg(\frac{\norm{\nabla f_\delta(y;\xi) - \nabla f_\delta(y)}^2_*}{\sigma^2}\Bigg)\right] \leq \exp(1),\,\,\, \forall y \in Q,
\end{gather}
где $\sigma^2 > 0$.
\end{definition}
Будем предполагать, что для $f(x)$ существует стохастический $(\delta, L)$-оракул.

\begin{statement}

Возьмем $x = y$ в (\ref{exitLDLSTOrig}), тогда
\begin{gather}
\label{exitLDLOrig2}
f_\delta(y) \leq f(y) \leq f_\delta(y) + \delta,\,\,\,\forall y \in Q.
\end{gather}
\end{statement}

Определим константу $R_Q$ такую, что
\begin{align}
R_Q \geq \max_{x,y \in Q}\norm{x - y}.
\end{align} Будем считать, что $R_Q < \infty$.

В методе, который мы предложим далее, будем оценивать истинный градиент на каждом шаге с помощью некоторого количества $\nabla f_\delta(y;\xi_j),\, j \in [1\dots m_{k+1}]$, используя технику mini-batch.

Обозначим
\begin{gather}
\widetilde{\nabla}^{m_{k+1}} f_\delta(y) \myeq \frac{1}{m_{k+1}}\sum_{j=1}^{m_{k+1}}\nabla f_\delta(y;\xi_j).
\end{gather}

Приведем важную лемму:
\begin{lemma}
\label{corST2}

Пусть $\left(f_\delta(y), \nabla f_\delta(y;\xi_i)\right),\,i = 1,\dots,m_{k+1}\,$ - $m_{k+1}$ независимых выхода стохастического $(\delta, L)$-оракула, $x, y \in Q$ -- случайные векторы, $y$ и $\xi_i\,\,i = 1,\dots,m_{k+1}$ -- независимы, $\widetilde{L}$ случайная константа, такая, что $\widetilde{L} \geq 3L/2$ почти наверное, $\kappa$ -- константа регулярности\footnote{Можно показать: если $\norm{\cdot}_* = \norm{\cdot}_2$, то $\kappa = 1$. Более того, если $\norm{\cdot}_* = \norm{\cdot}_q$, $q \in [2,\infty]$, то $\kappa = \min\left[q-1,2\ln(n)\right]$.} из \cite{juditsky2008large} для $(\mathds{R}^n,\norm{\cdot}_*)$ и выбрано произвольное $\Omega \geq 0$, тогда
\begin{align*}
&\hphantom{{}={}}\mathbb{P}\Bigg(f_\delta(x) - f_\delta(y) - \langle\widetilde{\nabla}^{m_{k+1}} f_\delta(y), x - y\rangle > (2\kappa + 4\Omega\sqrt{\kappa} + 2\Omega^2)\frac{3\sigma^2}{\widetilde{L}m_{k+1}} + \frac{\widetilde{L}}{2}\norm{x - y}^2 + \delta\Bigg) \\&\leq \exp(-\Omega^2/3).
\end{align*}
\end{lemma}

\begin{proof}
Рассмотрим правую часть условия (\ref{exitLDLSTOrig}):
\begin{gather*}
f(x) - f_\delta(y) - \langle\nabla f_\delta(y), x - y\rangle \leq \frac{L}{2}\norm{x - y}^2 + \delta.
\end{gather*}
Учтем (\ref{exitLDLOrig2}), тогда
\begin{gather*}
f_\delta(x) - f_\delta(y) - \langle\nabla f_\delta(y), x - y\rangle \leq \frac{L}{2}\norm{x - y}^2 + \delta.
\end{gather*}
В обе части неравенства добавим слагаемое $\langle\widetilde{\nabla}^{m_{k+1}} f_\delta(y), x - y\rangle$:
\begin{flalign*}
f_\delta(x) - f_\delta(y) - \langle\widetilde{\nabla}^{m_{k+1}} f_\delta(y), x - y\rangle &\leq \langle\nabla f_\delta(y) - \widetilde{\nabla}^{m_{k+1}} f_\delta(y), x - y\rangle + \frac{L}{2}\norm{x - y}^2 + \delta&\\ &\leq \langle\nabla f_\delta(y) - \widetilde{\nabla}^{m_{k+1}} f_\delta(y), x - y\rangle + \frac{\widetilde{L}}{3}\norm{x - y}^2 + \delta.
\end{flalign*}
Воспользуемся неравенством Фенхеля \cite{devolder2013exactness} (формула (7.6))
\begin{flalign*}
f_\delta(x) - f_\delta(y) - \langle\widetilde{\nabla}^{m_{k+1}} f_\delta(y), x - y\rangle &\leq \frac{\widetilde{L}}{6}\norm{x - y}^2 + \frac{3}{\widetilde{L}}\norm{\widetilde{\nabla}^{m_{k+1}} f_\delta(y) - \nabla f_\delta(y)}^2_* \\&\hphantom{{}={}}+ \frac{\widetilde{L}}{3}\norm{x - y}^2 + \delta.
\end{flalign*}
То есть
\begin{flalign}
\label{ST_H1}
f_\delta(x) - f_\delta(y) - \langle\widetilde{\nabla}^{m_{k+1}} f_\delta(y), x - y\rangle &\leq \frac{\widetilde{L}}{2}\norm{x - y}^2 + \frac{3}{\widetilde{L}}\norm{\widetilde{\nabla}^{m_{k+1}} f_\delta(y) - \nabla f_\delta(y)}^2_* + \delta.
\end{flalign}
Оценим вероятность того, что
\begin{equation}
\begin{gathered}
\label{ST_H2}
f_\delta(x) - f_\delta(y) - \langle\widetilde{\nabla}^{m_{k+1}} f_\delta(y), x - y\rangle > (2\kappa + 4\Omega\sqrt{\kappa} + 2\Omega^2)\frac{3\sigma^2}{Lm_{k+1}} + \frac{\widetilde{L}}{2}\norm{x - y}^2 + \delta.
\end{gathered}
\end{equation}
Учитывая (\ref{ST_H1}), из (\ref{ST_H2}) будет следовать
\begin{gather}
\label{ST_H3}
\frac{3}{\widetilde{L}}\norm{\widetilde{\nabla}^{m_{k+1}} f_\delta(y) - \nabla f_\delta(y)}^2_* > (2\kappa + 4\Omega\sqrt{\kappa} + 2\Omega^2)\frac{3\sigma^2}{\widetilde{L}m_{k+1}},
\end{gather}
что эквивалентно
\begin{gather}
\label{ST_H4}
\norm{\widetilde{\nabla}^{m_{k+1}} f_\delta(y) - \nabla f_\delta(y)}^2_* > (2\kappa + 4\Omega\sqrt{\kappa} + 2\Omega^2)\frac{\sigma^2}{m_{k+1}}.
\end{gather}
Воспользуемся следующим фактом из \cite{juditsky2008large}, пусть $\gamma_1,\dots,\gamma_N$ -- случайные независимые вектора такие, что
\begin{gather*}
\mathbb{E}\left[\exp\Big(\frac{\norm{\gamma_i}_*^2}{\sigma^2}\Big)\right] \leq \exp(1),\,\,\,\,\mathbb{E}\gamma_i = 0,
\end{gather*}
тогда верно, $\forall \Omega \geq 0$:
\begin{gather*}
\mathbb{P}\Bigg(\norm{\sum_{i=1}^{N}\gamma_i}_* \geq (\sqrt{2\kappa} + \sqrt{2}\Omega)\sqrt{N}\sigma\Bigg) \leq \exp(-\Omega^2/3),
\end{gather*}
где $\kappa$ -- константа регулярности из \cite{juditsky2008large}.
Возьмем $\gamma_i = \nabla f_\delta(\widetilde{y};\xi_i) - \nabla f_\delta(\widetilde{y})$, где $\widetilde{y}$ -- неслучайный вектор, учтем (\ref{ST1}) и (\ref{ST2}), тогда
\begin{gather*}
\mathbb{P}\Bigg(\norm{\sum_{j=1}^{m_{k+1}}\Big(\nabla f_\delta(\widetilde{y};\xi_j) - \nabla f_\delta(\widetilde{y})\Big)}_* > (\sqrt{2\kappa} + \sqrt{2}\Omega)\sqrt{m_{k+1}}\sigma\Bigg) \leq \exp(-\Omega^2/3)\Leftrightarrow\\
\mathbb{P}\Bigg(\norm{\widetilde{\nabla}^{m_{k+1}} f_\delta(\widetilde{y}) - \nabla f_\delta(\widetilde{y})}_* > (\sqrt{2\kappa} + \sqrt{2}\Omega)\frac{\sigma}{\sqrt{m_{k+1}}}\Bigg) \leq \exp(-\Omega^2/3).
\end{gather*}
Получаем в конечном счете, что
\begin{align*}
&\hphantom{{}={}}\mathbb{P}\Bigg(\norm{\widetilde{\nabla}^{m_{k+1}} f_\delta(y) - \nabla f_\delta(y)}_* > (\sqrt{2\kappa} + \sqrt{2}\Omega)\frac{\sigma}{\sqrt{m_{k+1}}}\Bigg) \\ &=
\mathbb{E}\Bigg[\mathbb{P}\Bigg(\norm{\widetilde{\nabla}^{m_{k+1}} f_\delta(\widetilde{y}) - \nabla f_\delta(\widetilde{y})}_* > (\sqrt{2\kappa} + \sqrt{2}\Omega)\frac{\sigma}{\sqrt{m_{k+1}}}\Bigg|y = \widetilde{y}\Bigg)\Bigg] \\&\leq
\mathbb{E}\left[\exp(-\Omega^2/3)\right] \\&= \exp(-\Omega^2/3)
\end{align*}
и
\begin{gather*}
\mathbb{P}\Bigg(\frac{3}{\widetilde{L}}\norm{\widetilde{\nabla}^{m_{k+1}} f_\delta(y) - \nabla f_\delta(y)}_*^2 > (2\kappa + 4\Omega\sqrt{\kappa} + 2\Omega^2)\frac{3\sigma^2}{\widetilde{L}m_{k+1}}\Bigg) \leq \exp(-\Omega^2/3).
\end{gather*}
Из этого неравенства и из того, что из (\ref{ST_H2}) следует (\ref{ST_H3}), получаем утверждение леммы.
\end{proof}

Сделаем следующее обозначение: $\widetilde{\Omega} \myeq 2\kappa + 4\Omega\sqrt{\kappa} + 2\Omega^2$. 

Далее будет использоваться в алгоритме константу $L_0$, которая имеет смысл предположительной "локальной"\, константы Липшица градиента в точке $x_0$. Рассмотрим алгоритм оптимизации.

\begin{mdframed}

\begin{mdframed}
\centering
\bf{Адаптивный зеркальный вариант метода подобных треугольников со стохастическим $(\delta, L)$-оракулом}
\end{mdframed}

\textbf{Дано:} $x_0$ -- начальная точка, $\epsilon$ -- желаемая точность решения, $\delta$, $L$ -- константы из $(\delta, L)$-оракула, $\beta$ -- доверительный уровень, $L_0$ ($L_0 \leq L$).

\iftrue
Возьмем 
\begin{gather*}
N := \left\lceil\frac{2\sqrt{3}\sqrt{L}R_Q}{\sqrt{\epsilon}}\right\rceil, \,
\Omega := \sqrt{6\ln{\frac{N}{\beta}}}.
\end{gather*}

\textbf{0 - шаг:}
\begin{gather*}
y_0 := x_0,\,
u_0 := x_0,\,
L_1^0 := \frac{L_0}{2},\,
\alpha_0 := 0,\,
A_0 := \alpha_0,\,
j_1 := 0
\end{gather*}

\textbf{$\boldsymbol{k+1}$ - шаг:}
\begin{gather}
\alpha_{k+1} := \frac{1 + \sqrt{1 + 4A_kL_{k+1}^{j_{k+1}}}}{2L_{k+1}^{j_{k+1}}} \label{eqymir2DLSTA}\\
A_{k+1} := A_k + \alpha_{k+1}\\
y_{k+1} := \frac{\alpha_{k+1}u_k + A_k x_k}{A_{k+1}} \label{eqymir2DLST}\\
m_{k+1} := \left\lceil\frac{3\sigma^2\widetilde{\Omega}\alpha_{k+1}}{\epsilon}\right\rceil
\label{eqymir2DLSTMK}
\end{gather}
\begin{center}
Сгенерировать: $ \widetilde{\nabla}^{m_{k+1}} f_\delta(y_{k+1})$
\end{center}
\begin{equation}
\begin{gathered}
\phi_{k+1}(x) \myeq V(x, u_k) + \alpha_{k+1}\left(f_\delta(y_{k+1}) + \langle \widetilde{\nabla}^{m_{k+1}} f_\delta(y_{k+1}), x - y_{k+1} \rangle + h(x)\right)\\
u_{k+1} := \argmin_{x \in Q}\phi_{k+1}(x) \label{equmir2DLST}
\end{gathered}
\end{equation}
\begin{equation}
x_{k+1} := \frac{\alpha_{k+1}u_{k+1} + A_k x_k}{A_{k+1}} \label{eqxmir2DLST}
\end{equation}
Если выполнено условие
\begin{equation}
\begin{gathered}
f_\delta(x_{k+1}) \leq f_\delta(y_{k+1}) + \langle \widetilde{\nabla}^{m_{k+1}} f_\delta(y_{k+1}), x_{k+1} - y_{k+1} \rangle\ +\\  + \frac{L_{k+1}^{j_{k+1}}}{2}\norm{x_{k+1} - y_{k+1}}^2 + \frac{3\sigma^2\widetilde{\Omega}}{L_{k+1}^{j_{k+1}}m_{k+1}} + \delta,
\label{exitLDLST}
\end{gathered}
\end{equation}
то
\begin{gather}
L_{k+2}^0 := \frac{L_{k+1}^{j_{k+1}}}{2},\,\,\,L_{k+1} := L_{k+1}^{j_{k+1}},\,\,\,j_{k+2} := 0
\end{gather}
и перейти к следующему шагу, иначе
\begin{gather}
L_{k+1}^{j_{k+1} + 1} := 2L_{k+1}^{j_{k+1}},\,\,\,j_{k+1} := j_{k+1} + 1
\end{gather}
и повторить текущий шаг.
\end{mdframed}

\begin{lemma}
\leavevmode
\label{remark2ST}
Пусть \begin{gather*}\Omega = \sqrt{6\ln{\frac{N}{\beta}}},\end{gather*} тогда
 \begin{align*}
 \mathbb{P}\left(\bigcap_{k=0}^{N-1}\left\{L_{k+1} < 3L\right\}\right) \geq 1 - \beta.
 \end{align*}

\end{lemma}
\begin{proof}

\begin{align*}
\mathbb{P}\left(\bigcup_{k=0}^{N-1}\left\{L_{k+1} \geq 3L\right\}\right) &\leq_{{\tiny \circled{1}}} \sum_{k = 0}^{N - 1}\mathbb{P}\left(L_{k+1} \geq 3L\right).
\end{align*}

{\small \circled{1}} -- из неравенства Бонферрони.

Обозначим за $A(k+1, j)$ -- событие, заключающиеся в выполнимости условия (\ref{exitLDLST}) на $k+1$ шаге на $j$-ом внутренним цикле. Рассмотрим для случая, когда $k = 0$. Из события \begin{align*}\left\{L_{1} \geq 3L\right\}\end{align*} следует событие \begin{align*}\left\{L_1^{j_1 - 1} \geq \frac{3L}{2} \cap \neg A(1,j_1 - 1)\right\}.\end{align*} 
Поэтому
\begin{align*}
\mathbb{P}\left(L_{1} \geq 3L\right)
&\leq \mathbb{P}\left(L_1^{j_1 - 1} \geq \frac{3L}{2} \cap \neg A(1,j_1 - 1)\right)\\
&\leq \mathbb{P}\left(\neg A(1,j_1 - 1)|L_1^{j_1 - 1} \geq \frac{3L}{2}\right)\\
&\leq_{{\tiny \circled{1}}} \exp(-\Omega^2/3).
\end{align*}
{\small \circled{1}} -- Лемма \ref{corST2}.

Воспользуемся методом математической индукции. Пусть для $k$ верно
\begin{align*}
\mathbb{P}\left(L_{k} \geq 3L\right) \leq k\exp(-\Omega^2/3),
\end{align*}

тогда
\begin{align*}
\mathbb{P}\left(L_{k+1} \geq 3L\right) &= \mathbb{P}\left(L_{k+1} \geq 3L \cap L_{k} \geq 3L\right) + \mathbb{P}\left(L_{k+1} \geq 3L \cap L_{k} < 3L\right)\\
&= \mathbb{P}\left(L_{k+1} \geq 3L |L_{k} \geq 3L\right)\mathbb{P}\left(L_{k} \geq 3L\right) + \mathbb{P}\left(L_{k+1} \geq 3L \cap L_{k} < 3L\right)\\
&\leq k\exp(-\Omega^2/3) + \mathbb{P}\left(L_{k+1} \geq 3L \cap L_{k} < 3L\right)\\
&\leq k\exp(-\Omega^2/3) + \mathbb{P}\left(L_{k+1} \geq 3L \cap L_{k+1}^0 < \frac{3L}{2}\right).
\end{align*}
Из события \begin{align*}\left\{L_{k+1} \geq 3L \cap L_{k+1}^0 < \frac{3L}{2}\right\}\end{align*} следует событие \begin{align*}\left\{L_{k+1}^{j_{k+1} - 1} \geq \frac{3L}{2} \cap \neg A(k+1,j_{k+1} - 1)\right\}.\end{align*}
То есть
\begin{align*}
\mathbb{P}\left(L_{k+1} \geq 3L\right) 
&\leq k\exp(-\Omega^2/3) + \mathbb{P}\left(L_{k+1}^{j_{k+1} - 1} \geq \frac{3L}{2} \cap \neg A(k+1,j_{k+1} - 1)\right)\\
&\leq k\exp(-\Omega^2/3) + \mathbb{P}\left(\neg A(k+1,j_{k+1} - 1) | L_{k+1}^{j_{k+1} - 1} \geq \frac{3L}{2}\right)\\
&\leq_{{\tiny \circled{1}}} (k+1)\exp(-\Omega^2/3).
\end{align*}
{\small \circled{1}} -- Лемма \ref{corST2}.

В конечном счете получаем, что
\begin{align*}
\mathbb{P}\left(\bigcup_{k=0}^{N-1}\left\{L_{k+1} \geq 3L\right\}\right) 
&\leq \sum_{k = 0}^{N - 1}\mathbb{P}\left(L_{k+1} \geq 3L\right)\\
&\leq \sum_{k = 0}^{N - 1}(k+1)\exp(-\Omega^2/3)\\
&\leq N^2\exp(-\Omega^2/3).
\end{align*}

Так как по условию леммы $\Omega = \sqrt{6\ln{\frac{N}{\beta}}}$, то
\begin{align*}
\mathbb{P}\left(\bigcup_{k=0}^{N-1}\left\{L_{k+1} \geq 3L\right\}\right) &\leq N^2\exp(-\Omega^2/3)\\ &\leq \beta.
\end{align*}
Последнее неравенство завершает доказательство леммы.
\end{proof}

\begin{lemma}
	\label{lemma_maxmin_DLST1}
	Пусть для последовательности $\alpha_k$ выполнено
	\begin{align*}
	\alpha_0 = 0,\,
	A_k = \sum_{i = 0}^{k}\alpha_i,\,
	\alpha_{k+1} = \frac{1 + \sqrt{1 + 4A_kL_{k+1}}}{2L_{k+1}},
	\end{align*}
	где $\{L_k\}$ -- последовательность, генерируемая алгоритмом.
	
	Тогда c вероятностью $1 - \beta$
	 \begin{align}
	 \label{lemma_maxmin_DLST1_1}
	 A_k \geq \frac{(k+1)^2}{12L},\,\,\,\forall k = 1,\dots,N.
	 \end{align}
\end{lemma}

\begin{proof}
Воспользуемся Леммой \ref{remark2ST}, которая говорит следующее: с вероятностью больше или равной $1 - \beta$, $\forall k \leq N$, $L_k < 3L$.
 
Пусть $k = 1$, тогда
	\begin{equation*}
	A_1 = \alpha_1 = \frac{1}{L_1} \geq \frac{1}{3L}.
	\end{equation*}
	По индукции, пусть неравенство (\ref{lemma_maxmin_DLST1_1}) верно для $k$, тогда:
	\begin{align*}
	\alpha_{k+1} &= \frac{1}{2L_{k+1}} + \sqrt{\frac{1}{4L_{k+1}^2} + \frac{A_{k}}{L_{k+1}}} \\&\geq 
	\frac{1}{2L_{k+1}} + \sqrt{\frac{A_{k}}{L_{k+1}}} \\&\geq
	\frac{1}{6L} + \frac{1}{\sqrt{3L}}\frac{k+1}{2\sqrt{3L}} \\&=
	\frac{k+2}{6L}
	\end{align*}
Последнее неравенство следует из индукционного предположения. В конечном счете получаем, что
	\begin{equation*}
	\alpha_{k+1} \geq \frac{k+2}{6L}
	\end{equation*}
и
	\begin{align*}
	A_{k+1} &= A_k + \alpha_{k+1} \\&= \frac{(k+1)^2}{12L} + \frac{k+2}{6L} \\&\geq \frac{(k+2)^2}{12L}.
	\end{align*}
\end{proof}

\begin{remark}
Как и в \cite{gasnikov2018universal,nesterov2015universal}, можно получить, что с вероятностью $1 - \beta$ в среднем на каждой внешней итерации мы будем считать значение всех функций 4 раза, а стохастического градиента $\widetilde{\nabla}^{m_{k+1}} f_\delta(y_{k+1})$ -- 2 раза. 
\end{remark}

Докажем важную лемму.

\begin{lemma}
	Пусть $\psi(x)$ выпуклая функция и 
	\begin{gather*}
	y = {\argmin_{x \in Q}}\{\psi(x) + V(x,z)\}.
	\end{gather*}
	Тогда

	\begin{equation*}
	\psi(x) + V(x,z) \geq \psi(y) + V(y,z) + V(x,y) ,\,\,\, \forall x \in Q.
	\end{equation*}
	\label{lemma_maxmin_2}
\end{lemma}

\begin{proof}
	
	По критерию оптимальности:
	\begin{gather*}
		\exists g \in \partial\psi(y), \,\,\, \langle g + \nabla_y V(y, z), x - y  \rangle \geq 0 ,\,\,\, \forall x \in Q.
	\end{gather*}
	Тогда неравенство
	\begin{gather*}
		\psi(x) - \psi(y) \geq \langle g, x - y  \rangle \geq \langle \nabla_y V(y, z), y - x  \rangle
	\end{gather*}
и равенство
	\begin{gather*}
	\langle \nabla_y V(y, z), y - x  \rangle = \langle \nabla d(y) - \nabla d(z), y - x  \rangle = d(y) - d(z) - \langle \nabla d(z), y - z  \rangle +\\ + d(x) - d(y) - \langle \nabla d(y), x - y  \rangle - d(x) + d(z) + \langle \nabla d(z), x - z  \rangle = \\=
	V(y,z) + V(x,y) - V(x,z)
	\end{gather*}
завершают доказательство.
	
\end{proof}

Введем обозначение: $l_f^\delta(x;y) = f_\delta(y) + \langle \widetilde{\nabla}^{m_{k+1}} f_\delta(y), x - y \rangle$.
\begin{lemma}
	С вероятностью больше или равной $1 - \beta$, $\forall x \in Q$, $\forall k \in [0,N]$,
	\begin{align*}
		&\hphantom{{}={}}l_f^\delta(x_{k+1};y_{k+1})  + \frac{L_{k+1}}{2}\norm{x_{k+1} - y_{k+1}}^2 + h(x_{k+1}) \\&\leq \frac{A_k}{A_{k+1}}\left(l_f^\delta(x_k;y_{k+1}) + h(x_k)\right) +
			 \frac{\alpha_{k+1}}{A_{k+1}}\Bigl(l_f^\delta(x;y_{k+1}) \\&\hphantom{{}={}}+ h(x)
			 + \frac{1}{\alpha_{k+1}}V(x, u_k) - \frac{1}{\alpha_{k+1}}V(x, u_{k+1})\Bigl)
	\end{align*}
	\label{lemma_maxmin_3DLST}
\end{lemma}
\begin{proof}
	\begin{align*}
	&\hphantom{{}={}}l_f^\delta(x_{k+1};y_{k+1})  + \frac{L_{k+1}}{2}\norm{x_{k+1} - y_{k+1}}^2 + h(x_{k+1}) \\
	&\eqarg{(\ref{eqxmir2DLST})}l_f^\delta\left(\frac{\alpha_{k+1}u_{k+1} + A_k x_k}{A_{k+1}};y_{k+1}\right)  + \frac{L_{k+1}}{2}\norm{\frac{\alpha_{k+1}u_{k+1} + A_k x_k}{A_{k+1}} - y_{k+1}}^2\\
	&\hphantom{{}={}}+ h\left(\frac{\alpha_{k+1}u_{k+1} + A_k x_k}{A_{k+1}}\right)\\
	&\leqarg{(\ref{eqymir2DLST})} f_\delta(y_{k+1}) + \frac{\alpha_{k+1}}{A_{k+1}}\langle \widetilde{\nabla}^{m_{k+1}} f_\delta(y_{k+1}), u_{k+1} - y_{k+1} \rangle\ \\
	&\hphantom{{}={}}+\frac{A_k}{A_{k+1}}\langle \widetilde{\nabla}^{m_{k+1}} f_\delta(y_{k+1}), x_k - y_{k+1} \rangle  + \frac{L_{k+1} \alpha^2_{k+1}}{2 A^2_{k+1}}\norm{u_{k+1} - u_k}^2 \\
	 &\hphantom{{}={}}+ \frac{\alpha_{k+1}}{A_{k+1}}h(u_{k+1}) + \frac{A_k}{A_{k+1}}h(x_k)\\&=
	 \frac{A_k}{A_{k+1}}\left(f_\delta(y_{k+1}) + \langle \widetilde{\nabla}^{m_{k+1}} f_\delta(y_{k+1}), x_k - y_{k+1} \rangle + h(x_k)\right)
	 \\&\hphantom{{}={}}+
	 \frac{\alpha_{k+1}}{A_{k+1}}\left(f_\delta(y_{k+1}) + 
	 \langle \widetilde{\nabla}^{m_{k+1}} f_\delta(y_{k+1}), u_{k+1} - y_{k+1} \rangle + h(u_{k+1})\right)\\
	 &\hphantom{{}={}}+\frac{L_{k+1} \alpha^2_{k+1}}{2 A^2_{k+1}}\norm{u_{k+1} - u_k}^2\\ &=_{{\tiny \circled{1}}}
	 \frac{A_k}{A_{k+1}}\left(l_f^\delta(x_k;y_{k+1}) + h(x_k)\right)\\
	 &\hphantom{{}={}}+\frac{\alpha_{k+1}}{A_{k+1}}\left(l_f^\delta(u_{k+1};y_{k+1})
	 + \frac{1}{2 \alpha_{k+1}}\norm{u_{k+1} - u_k}^2 + h(u_{k+1})\right) \\&\leqarg{(\ref{bregmanIneq})}
	 \frac{A_k}{A_{k+1}}\left(l_f^\delta(x_k;y_{k+1}) + h(x_k)\right)\\
	 &\hphantom{{}={}}+\frac{\alpha_{k+1}}{A_{k+1}}\left(l_f^\delta(u_{k+1};y_{k+1})
	 + \frac{1}{\alpha_{k+1}}V(u_{k+1}, u_k) + h(u_{k+1})\right) \\&\leq_{{\tiny \circled{2}}}
	 \frac{A_k}{A_{k+1}}\left(l_f^\delta(x_k;y_{k+1}) + h(x_k)\right) \\
	 &\hphantom{{}={}}+
	 \frac{\alpha_{k+1}}{A_{k+1}}\left(l_f^\delta(x;y_{k+1}) + h(x)
	 + \frac{1}{\alpha_{k+1}}V(x, u_k) - \frac{1}{\alpha_{k+1}}V(x, u_{k+1})\right).
	\end{align*}
\end{proof}

{\small \circled{1}} -- из $A_{k+1} = L_{k+1}\alpha^2_{k+1}$, это следует из (\ref{eqymir2DLSTA}).

{\small \circled{2}} -- из Леммы \ref{lemma_maxmin_2} с 
$\psi(x) = \alpha_{k+1}\left(f_\delta(y_{k+1}) + 
\langle \widetilde{\nabla}^{m_{k+1}} f_\delta(y_{k+1}), x - y_{k+1} \rangle + h(x)\right)$.

\begin{lemma}
	С вероятностью больше или равной $1 - \beta$, $\forall x \in Q$, $\forall k \in [0,N]$,
	\begin{align*}
		&\hphantom{{}={}}A_{k+1} F(x_{k+1}) - A_{k} F(x_{k}) + V(x, u_{k+1}) - V(x, u_{k}) \\&\leq \alpha_{k+1}F(x) + 2\delta A_{k+1} + \frac{3\sigma^2\widetilde{\Omega}}{L_{k+1}m_{k+1}}A_{k+1}+ \alpha_{k+1}\langle \widetilde{\nabla}^{m_{k+1}} f_\delta(y_{k+1})-\nabla f_\delta(y_{k+1}), x - u_k \rangle.
	\end{align*}
	\label{lemma_maxmin_3DLST_2}
\end{lemma}
\begin{proof}
\begin{align*}
F(x_{k+1}) &\leqarg{(\ref{exitLDLST}),(\ref{exitLDLOrig2})} l_{f_\delta}(x_{k+1};y_{k+1})  + \frac{L_{k+1}}{2}\norm{x_{k+1} - y_{k+1}}^2 + h(x_{k+1}) + \frac{3\sigma^2\widetilde{\Omega}}{L_{k+1}m_{k+1}} + 2\delta \\
&\leqarg{Лемма \ref{lemma_maxmin_3DLST}}\frac{A_k}{A_{k+1}}\left(l_{f_\delta}(x_k;y_{k+1}) + h(x_k)\right) + \frac{\alpha_{k+1}}{A_{k+1}}\Bigl(l_{f_\delta}(x;y_{k+1}) + h(x)
	 \\&\hphantom{{}={}}\,\,\,\,\,+ \frac{1}{\alpha_{k+1}}V(x, u_k) - \frac{1}{\alpha_{k+1}}V(x, u_{k+1})\Bigl) + \frac{3\sigma^2\widetilde{\Omega}}{L_{k+1}m_{k+1}} + 2\delta.
\end{align*}
Далее получаем, что
\begin{align*}
 F(x_{k+1}) &\leq \frac{A_k}{A_{k+1}}\left(f_\delta(y_{k+1}) + \langle \widetilde{\nabla}^{m_{k+1}} f_\delta(y_{k+1}), x_k - y_{k+1} \rangle + h(x_k)\right) \\&\hphantom{{}={}}+
\frac{\alpha_{k+1}}{A_{k+1}}\Bigl(f_\delta(y_{k+1}) + \langle \widetilde{\nabla}^{m_{k+1}} f_\delta(y_{k+1}), x - y_{k+1} \rangle + h(x)
	 \\&\hphantom{{}={}}+ \frac{1}{\alpha_{k+1}}V(x, u_k) - \frac{1}{\alpha_{k+1}}V(x, u_{k+1})\Bigl) + \frac{3\sigma^2\widetilde{\Omega}}{L_{k+1}m_{k+1}} + 2\delta\\&= \frac{A_k}{A_{k+1}}\Bigl(f_\delta(y_{k+1}) + \langle \nabla f_\delta(y_{k+1}), x_k - y_{k+1} \rangle+ h(x_k) \\&\hphantom{{}={}}+ \langle \widetilde{\nabla}^{m_{k+1}} f_\delta(y_{k+1})-\nabla f_\delta(y_{k+1}), x_k - y_{k+1} \rangle\Bigl) \\&\hphantom{{}={}}+
\frac{\alpha_{k+1}}{A_{k+1}}\Bigl(f_\delta(y_{k+1}) + \langle \nabla f_\delta(y_{k+1}), x - y_{k+1} \rangle + h(x)\\&\hphantom{{}={}}+ \langle \widetilde{\nabla}^{m_{k+1}} f_\delta(y_{k+1})-\nabla f_\delta(y_{k+1}), x - y_{k+1} \rangle
	 \\&\hphantom{{}={}}+ \frac{1}{\alpha_{k+1}}V(x, u_k) - \frac{1}{\alpha_{k+1}}V(x, u_{k+1})\Bigl) + \frac{3\sigma^2\widetilde{\Omega}}{L_{k+1}m_{k+1}} + 2\delta \\&\leq_{{\tiny \circled{1}}}
	 \frac{A_k}{A_{k+1}}F(x_k) + \frac{\alpha_{k+1}}{A_{k+1}}\Bigl(F(x) + \frac{1}{\alpha_{k+1}}V(x, u_k) - \frac{1}{\alpha_{k+1}}V(x, u_{k+1})\Bigl) \\&\hphantom{{}={}}+ \frac{3\sigma^2\widetilde{\Omega}}{L_{k+1}m_{k+1}} + 2\delta +\frac{\alpha_{k+1}}{A_{k+1}}\Bigl(\langle \widetilde{\nabla}^{m_{k+1}} f_\delta(y_{k+1})-\nabla f_\delta(y_{k+1}), x - y_{k+1} \rangle\Bigl) \\&\hphantom{{}={}}+ \frac{\alpha_{k+1}}{A_{k+1}}\langle \widetilde{\nabla}^{m_{k+1}} f_\delta(y_{k+1})-\nabla f_\delta(y_{k+1}), y_{k+1} - u_k \rangle
	 \\&=
	 	 \frac{A_k}{A_{k+1}}F(x_k) + \frac{\alpha_{k+1}}{A_{k+1}}\left(F(x) + \frac{1}{\alpha_{k+1}}V(x, u_k) - \frac{1}{\alpha_{k+1}}V(x, u_{k+1})\right) \\&\hphantom{{}={}}+ \frac{3\sigma^2\widetilde{\Omega}}{L_{k+1}m_{k+1}} + 2\delta +\frac{\alpha_{k+1}}{A_{k+1}}\left(\langle \widetilde{\nabla}^{m_{k+1}} f_\delta(y_{k+1})-\nabla f_\delta(y_{k+1}), x - u_{k} \rangle\right)
\end{align*}

{\small \circled{1}} -- из правой части (\ref{exitLDLOrig2}), Следствия \ref{exitLDLOrig2} и $A_{k}(y_{k+1} - x_k) = \alpha_{k+1} (u_k - y_{k+1})$ из (\ref{eqymir2DLST}).

\end{proof}

Нам будет полезен факт из \cite{lan2012validation, devolder2013exactness}.
\begin{lemma}
\label{lemmaDev}
Пусть $\gamma_1$,...,$\gamma_k$ -- i.i.d случайные величины, $\Gamma_k$ и $\nu_k$ -- неслучайные функции от $\gamma_i$, и $c_i$ -- неслучайные константы, для которых верно следующее
\begin{gather*}
\mathbb{E}\left[\Gamma_i|\gamma_1,\dots,\gamma_{i-1}\right] = 0,\\
\abs{\Gamma_i} \leq c_i \nu_i,\\
\mathbb{E}\left[\exp\Bigg(\frac{\nu_i^2}{\sigma^2}\Bigg)\Bigg|\gamma_1,\dots,\gamma_{i-1}\right] \leq \exp(1).
\end{gather*}
Тогда
\begin{gather*}
\mathbb{P}\left(\sum_{i=1}^{k}\Gamma_i \geq \sqrt{3}\sqrt{\widehat{\Omega}}\sigma\sqrt{\sum_{i=1}^{k}c_i^2}\right) \leq \exp(-\widehat{\Omega}),\,\,\,\forall k,\,\,\,\forall \widehat{\Omega} \geq 0.
\end{gather*}
\end{lemma}

\begin{lemma}
	\label{mainTheoremDLST}
	Пусть $V(x_*, x_0) \leq R^2$, где $x_0$ -- начальная точка, а $x_*$ -- ближайшая точка минимума к точке $x_0$ в смысле дивергенции Брэгмана, тогда с вероятностью $1 - 2\beta$
	\begin{equation*}
	F(x_N) - F(x_*)\leq \frac{R^2}{A_{N}} + 2\delta N + \epsilon + R_Q\sqrt{\frac{\epsilon}{A_N}}.
	\end{equation*}
\end{lemma}
\begin{proof}

	Учтем (\ref{eqymir2DLSTMK}) и (\ref{eqymir2DLSTA}) в Лемме \ref{lemma_maxmin_3DLST_2}, тогда с вероятностью большой или равной $1 - \beta$, $\forall k \geq 0$,
	\begin{align*}
		&\hphantom{{}={}}A_{k+1} F(x_{k+1}) - A_{k} F(x_{k}) + V(x, u_{k+1}) - V(x, u_{k}) \\&\leq \alpha_{k+1}F(x) + 2\delta A_{k+1} + \alpha_{k+1}\epsilon + \alpha_{k+1}\langle \widetilde{\nabla}^{m_{k+1}} f_\delta(y_{k+1})-\nabla f_\delta(y_{k+1}), x - u_k \rangle
	\end{align*}
	
	Просуммируем неравенства по $k = 0, ..., N - 1$, 
	\begin{align*}
	&\hphantom{{}={}}A_{N} F(x_N) - A_{0} F(x_0) + V(x, u_N) - V(x, u_0) \\&\leq (A_N - A_0)F(x) + 2\delta\sum_{k = 0}^{N-1}A_{k+1} + \sum_{k = 0}^{N-1}\alpha_{k+1}\epsilon \\&\hphantom{{}={}}+ \sum_{k = 0}^{N-1}\alpha_{k+1}\langle \widetilde{\nabla}^{m_{k+1}} f_\delta(y_{k+1})-\nabla f_\delta(y_{k+1}), x - u_k \rangle
	\end{align*}
	Откуда, с учетом неравенства $V(x, u_N)\geq 0,\,\forall x \in Q$,
	\begin{align*}
		A_{N} F(x_N) - A_NF(x) &\leq V(x, u_0) +
		 2\delta\sum_{k = 0}^{N-1}A_{k+1}+A_N\epsilon\\&\hphantom{{}={}}+ \sum_{k = 0}^{N-1}\alpha_{k+1}\langle \widetilde{\nabla}^{m_{k+1}} f_\delta(y_{k+1})-\nabla f_\delta(y_{k+1}), x - u_k \rangle.
	\end{align*}
	
	Возьмем $x = x_*$, оценим $A_{k+1}$ через $A_N$, тогда
	\begin{align*}
	A_{N} F(x_N) - A_NF(x_*)  &\leq V(x_*, u_0) +
			 2\delta N A_{N}+A_N\epsilon\\&\hphantom{{}={}}+ \sum_{k = 0}^{N-1}\alpha_{k+1}\langle \widetilde{\nabla}^{m_{k+1}} f_\delta(y_{k+1})-\nabla f_\delta(y_{k+1}), x - u_k \rangle.
	\end{align*}
	Поделим обе части на $A_{N}$:
	\begin{align}
	\label{beforeLemmause}
	F(x_N) - F(x_*)\leq \frac{R^2}{A_{N}} + 2\delta N + \epsilon + \sum_{k = 0}^{N-1}\frac{\alpha_{k+1}}{A_N}\langle \widetilde{\nabla}^{m_{k+1}} f_\delta(y_{k+1})-\nabla f_\delta(y_{k+1}), x - u_k \rangle.
	\end{align}
	
	Воспользуемся Леммой \ref{lemmaDev} для последнего слагаемого в неравенстве c \begin{gather*}
	\gamma_i =\xi_i,\\
	\Gamma_i = \frac{\alpha_{k_i+1}}{A_N m_{k_i+1}}\langle\nabla f_\delta(y_{k_i};\xi_i) - \nabla f_\delta(y_{k_i}), x - u_{(k_i - 1)}\rangle,\\
	 c_{i} = \frac{R_Q\alpha_{k_i+1}}{A_N m_{k_i+1}},\\
	 \nu_i = \norm{\nabla f_\delta(y_{k_i};\xi_i) - \nabla f_\delta(y_{k_i})}_*,\end{gather*} 
	$i \in [1,\dots,\sum_{k=0}^{N-1}m_{k+1}]$, где $k_i$ равно $k + 1$ для всех $i \in [m_{k} + 1,\dots, m_{k+1}]$. Выберем $\widehat{\Omega} = \ln(1/\beta) \leq \widetilde{\Omega}$, тогда с вероятностью не меньше $1 - \beta$
	\begin{align}
	\label{afterLemmause}
	\sum_{k = 0}^{N-1}\frac{\alpha_{k+1}}{A_N}\langle \widetilde{\nabla}^{m_{k+1}} f_\delta(y_{k+1})-\nabla f_\delta(y_{k+1}), x - u_k \rangle \leq \sqrt{3}\sigma\sqrt{\widehat{\Omega}}\sqrt{\sum_{i=0}^{N-1}\frac{R_Q^2\alpha^2_{k+1}}{A_N^2m_{k+1}}}.
	\end{align}
	Объединяя (\ref{beforeLemmause}) и (\ref{afterLemmause}) получаем, что с вероятностью не меньше $1 - 2\beta$
	\begin{gather*}
	F(x_N) - F(x_*)\leq \frac{R^2}{A_{N}} + 2\delta N + \epsilon + \sqrt{3}\sigma\sqrt{\widehat{\Omega}}\sqrt{\sum_{i=0}^{N-1}\frac{R_Q^2\alpha^2_{k+1}}{A_N^2m_{k+1}}}.
	\end{gather*}
	
	Учтем (\ref{eqymir2DLSTMK})
	\begin{align*}
	F(x_N) - F(x_*)&\leq \frac{R^2}{A_{N}} + 2\delta N + \epsilon + R_Q\frac{\sqrt{\widehat{\Omega}}}{\sqrt{\widetilde{\Omega}}}\sqrt{\sum_{i=0}^{N-1}\frac{\alpha_{k+1}\epsilon}{A_N^2}}\\
	&\leq \frac{R^2}{A_{N}} + 2\delta N + \epsilon + R_Q\sqrt{\frac{\epsilon}{A_N}}.
	\end{align*}
	
\end{proof}

\begin{theorem}
\label{mainTheoremDLST2}
Пусть \begin{gather*}\delta \leq \frac{\epsilon^\frac{3}{2}}{6\sqrt{3}\sqrt{L}R_Q}.\end{gather*} Тогда с вероятностью $1 - 3\beta$
\begin{gather*}
F(x_N) - F(x_*)\leq 4\epsilon.
	\end{gather*}
\end{theorem}

\begin{proof}
Мы знаем из Леммы \ref{lemma_maxmin_DLST1}, что с вероятностью $1 - \beta$ верно неравенство \begin{gather*}
A_N \geq \frac{(N+1)^2}{12L}.\end{gather*} 
Из последнего неравенства и условия на $N$ получаем, что
\begin{gather*}
\frac{R^2}{A_{N}} \leq \frac{R_Q^2}{A_{N}} \leq \frac{12LR_Q^2}{\left(\frac{2\sqrt{3}\sqrt{L}R_Q}{\sqrt{\epsilon}}+1\right)^2}\leq
\frac{R_Q^2\epsilon}{R_Q^2}=\epsilon
\end{gather*}
и
\begin{gather*}
R_Q\sqrt{\frac{\epsilon}{A_N}} \leq R_Q\sqrt{\frac{\epsilon^2}{R_Q^2}} =\epsilon.
\end{gather*}
Помимо этого с вероятностью $1 - 2\beta$ из Леммы \ref{mainTheoremDLST}
\begin{equation*}
F(x_N) - F(x_*)\leq \frac{R^2}{A_{N}} + 2\delta N + \epsilon + R_Q\sqrt{\frac{\epsilon}{A_N}}.
\end{equation*}
Объединяя все вместе, включая условие на $\delta$, получаем, что с вероятностью $1 - 3\beta$
\begin{equation*}
F(x_N) - F(x_*) \leq 4\epsilon.
\end{equation*}
\end{proof}

\begin{corollary}
\leavevmode
Далее все утверждения будут выполнятся с вероятностью не меньше $1 - 3\beta$. Оценим количество обращений $M$ к оракулу за стохастическими градиентами. По ходу алгоритма мы можем контролировать $R_Q^2/A_N$, тогда пусть $\widetilde{N} + 1$ -- минимальное число шагов, для которого выполнено $R_Q^2/A_{\widetilde{N} + 1} \leq \epsilon$. Ясно, что $\widetilde{N} + 1 \leq N$, причем условие $R_Q^2/A_{\widetilde{N} + 1} \leq \epsilon$ является достаточным условием для достижения $\epsilon$-решения по функции. Как и в \cite{nesterov2015universal,gasnikov2018universal} количество обращений за $\widetilde{\nabla}^{m_{k+1}} f_\delta(y_{k+1})$ на $k$-ом шаге будет равно $2 + \log\left(L_k/L_{k-1}\right)$. Также отметим, что $L_{\widetilde{N} + 1} \leq 3L$. Поэтому общее количество обращений к оракулу за $\nabla f_\delta(y;\xi)$ будет равно

\begin{align*}
M &= 2\sum_{k=1}^{\widetilde{N} + 1}m_{k}\left(2 + \log\left(\frac{L_k}{L_{k-1}}\right)\right) \\&= 4\sum_{k=1}^{\widetilde{N} + 1}m_{k} + 2\sum_{k=1}^{\widetilde{N} + 1}m_{k}\log\left(\frac{L_k}{L_{k-1}}\right) \\&\leq
4\sum_{k=1}^{\widetilde{N} + 1}m_{k} + 2\sum_{j=1}^{\widetilde{N} + 1}m_{j}\sum_{k=1}^{\widetilde{N} + 1}\log\left(\frac{L_k}{L_{k-1}}\right) \\&\leq
\left(4 + \log\left(\frac{L_{\widetilde{N} + 1}}{L_0}\right)\right)\sum_{k=1}^{\widetilde{N} + 1}m_{k} \\&\leq \left(4 + \log\left(\frac{3L}{L_0}\right)\right)\sum_{k=1}^{\widetilde{N} + 1}m_{k}
\end{align*}

Рассмотрим $\sum_{k=1}^{\widetilde{N} + 1}m_{k}$:

\begin{gather*}
\sum_{k=0}^{\widetilde{N}}m_{k+1}\leq {\widetilde{N} + 1} + \frac{3\sigma^2\widetilde{\Omega}}{\epsilon}A_{\widetilde{N} + 1}.
\end{gather*}

Так как
\begin{align*}
\alpha_{\widetilde{N}+1} &= \frac{1}{2L_{\widetilde{N}+1}} + \sqrt{\frac{1}{4L_{\widetilde{N}+1}^2} + \frac{A_{\widetilde{N}}}{L_{\widetilde{N}+1}}} \\&\leq_{{\tiny \circled{1}}} \frac{1}{L_{\widetilde{N}+1}} + \sqrt{\frac{A_{\widetilde{N}}}{L_{\widetilde{N}+1}}} \\&\leq \frac{2}{L_{\widetilde{N}}} + \sqrt{\frac{2A_{\widetilde{N}}}{L_{\widetilde{N}}}} \\&= \frac{2}{L_{\widetilde{N}}} + \sqrt{2}\alpha_{\widetilde{N}}.
\end{align*}

{\small \circled{1}} -- из $\sqrt{x + y} \leq \sqrt{x} + \sqrt{y}$.

Ко всему прочему
\begin{align*}
\alpha_{\widetilde{N}} &= \frac{1}{2L_{\widetilde{N}}} + \sqrt{\frac{1}{4L_{\widetilde{N}}^2} + \frac{A_{\widetilde{N} - 1}}{L_{\widetilde{N}}}} \\&\geq \frac{1}{2L_{\widetilde{N}}}.
\end{align*}
Отсюда
\begin{align*}
\alpha_{\widetilde{N} + 1} &\leq (4 + \sqrt{2})\alpha_{\widetilde{N}} \\&\leq (4 + \sqrt{2})A_{\widetilde{N}} \\&\leq 6A_{\widetilde{N}}.
\end{align*}
Так как \begin{align*}\frac{R_Q^2}{A_{\widetilde{N}}} > \epsilon,\end{align*} это следует из минимальности $\widetilde{N} + 1$, то \begin{align*}\frac{R_Q^2}{A_{\widetilde{N} + 1}} &= \frac{R_Q^2}{A_{\widetilde{N}} + \alpha_{\widetilde{N} + 1}} \\&\geq \frac{R_Q^2}{A_{\widetilde{N}} + 6A_{\widetilde{N}}} \\&> \frac{\epsilon}{7}\end{align*}
и
\begin{align*}
\sum_{k=0}^{\widetilde{N}}m_{k+1} &\leq {\widetilde{N} + 1} + \frac{3\sigma^2\widetilde{\Omega}}{\epsilon}A_{\widetilde{N} + 1} \\&\leq {\widetilde{N} + 1} + \frac{21\sigma^2\widetilde{\Omega}R_Q^2}{\epsilon^2}.
\end{align*}

В конечном счете

\begin{align*}
M  &\leq \left(4 + \log\left(\frac{3L}{L_0}\right)\right)\left({\widetilde{N} + 1} + \frac{21\sigma^2\widetilde{\Omega}R_Q^2}{\epsilon^2}\right) \\&\leq
\left(4 + \log\left(\frac{3L}{L_0}\right)\right)\left(\frac{2\sqrt{3}\sqrt{L}R_Q}{\sqrt{\epsilon}} + \frac{21\sigma^2\widetilde{\Omega}R_Q^2}{\epsilon^2} + 1\right).
\end{align*}
\end{corollary}

Данный результат является оптимальным \cite{lan2012optimal}, то есть в общем случае с точностью до константы нельзя сделать количество обращений за стохастическим градиентом меньше.

\begin{remark}
\leavevmode
Стоит отметить, что ситуация, когда мы можем точно посчитать значение функции (не делая сэмплирование) и неточно градиент, как это делается в Определении \ref{defdeltaLstoch}, вообще говоря нетипична \cite{baydin2015automatic,krizhevsky2012imagenet}. Как правило, сложность подсчета функции с точностью до мультипликативной константы равна сложности подсчета градиента, поэтому полностью подсчитать значение функции за разумное время нельзя во многих ситуациях. В таком случае нужно оценивать значение функции, используя технику mini-batch. В данной работе это не представлено, но по аналогии можно показать, что все оценки несильно изменятся. Более подробный анализ планируется сделать в будущих работах.   
\end{remark}

\section{Заключение}
В данной работе нам удалось предоставить адаптивный быстрый стохастический метод и привести полные доказательства для оценок в смысле больших уклонений. Стоит отметить, что данный подход может иметь дальнейшее развитие, в частности, в будущем планируется сделать из данного метода универсальный метод \cite{nesterov2015universal} и перенести его на более общие постановки с моделью \cite{tyurin2017fast}. Ко всему этому, мы верим, что данный подход может быть обобщен на произвольное множество.

Работа была поддержана грантом РНФ 17-11-01027. 

\bibliographystyle{plain}
\bibliography{article}

\end{document}